\documentclass[10pt,a4paper]{amsart}
\usepackage{amssymb}
 \usepackage[latin1]{inputenc}
 \usepackage[english]{babel}
 \usepackage{amsfonts}
 \usepackage{amsmath}
 \usepackage{amsthm}
 \usepackage[dvips]{graphicx}
 \DeclareGraphicsExtensions{.pdf,.png,.jpg,.eps}
 \usepackage{epstopdf}
 \usepackage{fancyhdr}
 \usepackage{colortbl}
 \usepackage[pdftex]{hyperref}
\numberwithin{equation}{section}
\newtheorem{theorem}{Theorem}[section]

\theoremstyle{remark}
\newtheorem{remark}{Remark}[section]

\theoremstyle{definition}

\newcommand{\triple}[1]{{|\!|\!|#1|\!|\!|}}
\newcommand{\R}{\mathbb{R}}

\newcommand{\xx}{\langle x\rangle}

\begin{document}


\title[$L^p$-$L^q$ estimates for Electromagnetic Helmholtz equation.]
{$L^p$-$L^q$ estimates for Electromagnetic Helmholtz equation.}

\author{Andoni Garcia}
\address{Andoni Garcia: Universidad del Pais Vasco, Departamento de
Matem$\acute{\text{a}}$ticas, Apartado 644, 48080, Bilbao, Spain}
\email{andoni.garcia@ehu.es}


\thanks{The author is supported by the grant BFI06.42 of the Basque
Government}

\begin{abstract}
In space dimension $n\geq3$, we consider the electromagnetic Schr\"odinger Hamiltonian $H=(\nabla-iA(x))^2-V$ and the corresponding Helmholtz equation
  \begin{equation*}
   (\nabla-iA(x))^2u+u-V(x)u=f\quad \text{in}\quad \mathbb{R}^n.
  \end{equation*}
We extend the well known $L^p$-$L^q$ estimates for the solution of the free Helmholtz equation to the case when the electromagnetic hamiltonian $H$ is considered. 
 \end{abstract}

\date{\today}


\maketitle

\section{Introduction}\label{sec:introd}
This paper is devoted to the study of some estimates for the Helmholtz equation with electromagnetic potential. A very natural question is to extend the known results for the Helmholtz equation with constant coefficients to the case when we consider the perturbed Helmholtz equation by a potential. 
Our goal will be to extend the well known $L^{p}$-$L^{q}$ estimates for the free Helmholtz equation given in \cite{KRS}, \cite{CS}, \cite{Gut} and \cite{Gut1} to the case when we perturb the equation with an electromagnetic potential. More precisely, conditions on the electric and the magnetic part of the potential will be given in order to ensure that the estimates remain true. 
The $L^{p}$-$L^{q}$ estimates for the free Helmholtz equation are the following:
\begin{equation}\label{eq:lpestimates}
\|u\|_{L^q(\mathbb{R}^n)}\leq C  \|f\|_{L^p(\mathbb{R}^n)},
\end{equation}
where $u$ is a solution of 
\begin{equation}
\Delta u+(\tau\pm i\epsilon)u=-f\quad \tau,\epsilon>0.
\end{equation}
The exponents $p$ and $q$ in \eqref{eq:lpestimates} have to verify some specific conditions that will be specified later on. Here $C$ can depend on $\tau$, $p$, $q$ and $n$ and is independent of $\epsilon$.

The investigation of the estimates \eqref{eq:lpestimates} started in \cite{KRS}, where the study of uniform Sobolev estimates for constant coefficient second order differential operators was accomplished. Later on, in \cite{Gut}, \cite{Gut1} (See also \cite{CS}) the range of the exponents $p$ and $q$ where the estimates \eqref{eq:lpestimates} hold was determined.

Moreover, in the purely electric case i.e., for Schr\"odinger  Hamiltonians of the type $\Delta+V(x)$, where $V:\R^{n}\to\R$ is the electric potential, some positive results were given in \cite{RV}.

Therefore the aim of the paper is to prove the corresponding estimates \eqref{eq:lpestimates} in the case when the electromagnetic Schr\"odinger hamiltonian is considered .

In the first part we will prove that the existing results for the free Helmholtz equation can be extended to the perturbed equation if we impose precise decay conditions at infinity for the electric and the magnetic potential. This can be done without assuming smallness, neither for the electric part, nor for the magnetic part.

As it will be shown, for the electromagnetic case, the range for the exponents $p$ and $q$ where the estimates \eqref{eq:lpestimates} are valid is not the same as the one for the free case, hence in order to go further, we deal with the Helmholtz equation with purely electric potential, trying to obtain results in the same region of boundedness of the free equation.

Therefore, we consider the electromagnetic Schr\"odinger hamiltonian $H$ of the form 
\begin{equation}
H=(\nabla-iA(x))^2-V(x),
\end{equation}
and the corresponding Helmholtz equation in dimensions $n\geq3$, namely,
\begin{equation}\label{eq:Helmmagneticin}
	(\nabla-iA(x))^2u+u-V(x)u=f\quad \text{in}\quad \mathbb{R}^n.
\end{equation}
Here  $A:(A^1,\dots, A^n):\mathbb{R}^n\rightarrow\mathbb{R}^n$ is the magnetic potential and $V(x):\mathbb{R}^n\rightarrow\mathbb{R}$ is the electric potential. Since now on, we denote by
\begin{equation*}
	\nabla_{A}=\nabla-iA,\qquad\Delta_{A}=\nabla_{A}^2.
\end{equation*}
The magnetic potential $A$ is a mathematical construction which describes the interaction of particles with an external magnetic field. The magnetic field $B$, which is the physically measurable quantity, is given by
\begin{equation}\label{eq:B}
  B\in\mathcal M_{n\times n},
  \qquad
  B=DA-(DA)^t,
\end{equation}
i.e. it is the anti-symmetric gradient of the vector field $A$ (or, in geometrical terms, the differential $dA$ of the 1-form which is standardly associated to $A$). In dimension $n=3$ the action of $B$ on vectors is identified with the vector field $\text{curl}A$, namely
\begin{equation}\label{eq:B3}
  Bv=\text{curl}A\times v
  \qquad
  n=3,
\end{equation}
where the cross denotes the vectorial product in $\R^3$.

One of the most interesting facts related to the  $L^{p}$-$L^{q}$ estimates for the electromagnetic hamiltonian is that it seems that in order to conclude the boundedness of the solution, one should be able to bound the first order term that appears when the hamiltonian $H$ is expanded. More concretely, when the term $(\nabla-iA(x))^2$ in \eqref{eq:Helmmagneticin} is expanded, a first orden term, namely $A\cdot\nabla$, comes out and it is well known that there are no $L^{p}$-$L^{q}$ estimates for the gradient of the solution of the free Helmholtz equation,
\begin{equation}
\Delta u+u=-f\quad \text{in}\quad \mathbb{R}^n.
\end{equation}

We will proceed in the following way. Let us consider the modified Helmholtz equation with electromagnetic potential and fixed frequency $\tau=1$. It reads as follows
\begin{equation}\label{eq:Helmagmo}
	(\nabla-iA(x))^2u+(1\pm i\epsilon)u-V(x)u=f\quad \text{in}\quad \mathbb{R}^n,\quad \epsilon\neq0.
\end{equation}
\begin{remark}
For convenience we will deal only with the case $\tau=1$, in contrast with the case of general $\tau>0$.
\end{remark}

We will prove the corresponding $L^{p}$-$L^{q}$ estimates, independent of $\epsilon$, for the solution of \eqref{eq:Helmagmo}. The independence of $\epsilon$  will imply that the result remains true for the solution of \eqref{eq:Helmmagneticin}. This can be seen in \cite{IS}.

Our method is a mixture of an a priori estimate and perturbative arguments. This is what allows us to avoid smallness conditions in the potentials. Similar arguments have been used in the setting of the free Schr\"odinger equation, as can be seen in \cite{BPST} and \cite{DFVV}. Along the proof our basic tools will be the corresponding  $L^{p}$-$L^{q}$ estimates and a $L^{2}$-local estimate for the solution of the free Helmholtz equation, together with an a priori estimate for the solution of the modified Helmholtz equation with electromagnetic potential \eqref{eq:Helmagmo}.

Before we describe the results that we are going to use during our proof, let us introduce some basic notation. For $f:\mathbb{R}^n\to\mathbb{C}$ we define the Morrey-Campanato norm as 
\begin{equation}\label{eq:triplenorma}
\triple{f}^2:=\sup_{R>0}\frac{1}{R}\int_{|x|\leq R}|f|^2dx.
\end{equation}
Moreover, we denote, for $j\in \mathbb{Z}$, the annulus C(j) by 
\begin{equation*}
C(j)=\{x\in\mathbb{R}^n:2^j\leq|x|\leq2^{j+1}\},
\end{equation*}
\begin{equation}\label{eq:N}
	N(f):=\sum_{j\in \mathbb{Z}}\left(2^{j+1}\int_{C(j)}|f|^2dx\right)^{1/2},
\end{equation}
and we easily see the duality relation
\begin{equation*}
\int fgdx\leq\triple{g}\cdot N(f).
\end{equation*}
These norms were introduced by Agmon and H\"ormander in \cite{AH}.

\noindent During the exposition, the truncated version of the norms appearing above will be necessary. We will denote them respectively by

\begin{equation}\label{eq:triplenorma0}
\triple{f}_{0}^2:=\sup_{R\geq1}\frac{1}{R}\int_{|x|\leq R}|f|^2dx,
\end{equation}
\begin{equation}\label{eq:N0}
	N_{0}(f):=\sum_{j\geq 0}\left(2^{j+1}\int_{C(j)}|f|^2dx\right)^{1/2}.
\end{equation}
Let us also denote by $L^2_{\beta}(\mathbb{R}^n)$, for $\beta\in\mathbb{R}$, the Hilbert space of all functions $f$ such that $(1+|x|)^\beta f$ is square integrable over $\mathbb{R}^n$. The norm in this space is denoted by $\|\cdot\|_{\beta}$. Trivially we have that if $\beta>1/2$ and $f\in L^2_{\beta}(\mathbb{R}^n)$, then $N_{0}(f)<+\infty$.

As we have said, part of our method is perturbative, so in order to be able to start, let us remind what is known for the free Helmholtz equation. Firstly, we are going to state the result concerning the $L^{p}$-$L^{q}$ estimates which appears in \cite{Gut} and \cite{Gut1}. Let be
\begin{align*}
&A=\left(\frac{n+3}{2n},\frac{n-1}{2n}\right),\qquad \qquad \quad A'=\left(\frac{n+1}{2n},\frac{n-3}{2n}\right)\\
&
B=\left(\frac{n^2+4n-1}{2n(n+1)},\frac{n-1}{2n}\right),\qquad B'=\left(\frac{n+1}{2n},\frac{n^2-2n+1}{2n(n+1)}\right)\nonumber\\
\end{align*}
and $\Delta(n)$, the set of points of $[0,1]\times[0,1]$ given by
\begin{equation}
	\Delta(n)=\left\{\left(\frac{1}{p},\frac{1}{q}\right)\in[0,1]^2:\frac{2}{n+1}\leq \frac{1}{p}-\frac{1}{q}\leq\frac{2}{n}, \frac{1}{p}>\frac{n+1}{2n}, \frac{1}{q}<\frac{n-1}{2n}\right\}.
\end{equation}

\noindent The set $\Delta(n)$ is the trapezium $A$$B$$B'$$A'$ with the closed line segments $AB$ and $B'A'$ removed (see Figure  \ref{fig:deltan}). 


\begin{figure}[!htb]
\begin{center}
\includegraphics[]{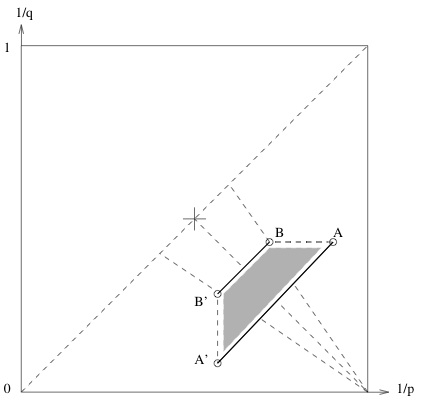}
\caption{$\Delta(n)$, $n\geq3$.}
\label{fig:deltan}
\end{center}
\end{figure}
\newpage
Now, we are in conditions to recall the existing result for the Helmholtz equation with constant coefficients. 
\begin{remark}
In \cite{Gut} (See also \cite{Gut1}), estimates for the solution of the equation perturbed with generals $\tau>0$ and $\epsilon>0$, are given, namely,
\begin{equation}
\Delta u+(\tau+i\epsilon)u=-F,\quad\tau,\epsilon>0.
\end{equation}
The special case where the point $(1/p,1/q)$ lies on the open segment $AA'$ and on the duality line $1/q=1-1/p$ in Figure \ref{fig:deltan} was previously obtained in [\cite{KRS}, Theorem 2.2 and 2.3 respectively].
\end{remark}
\noindent Recall that we will only deal with the case of fixed frequency $\tau=1$.
The Theorem reads as follows.

\begin{theorem}\label{thm:Helmfree}
	Let u be a solution of 
	\begin{equation}
		\Delta u+(1+i\epsilon)u=-F,\quad\epsilon>0.
	\end{equation}
	Then, there exists a constant C, independent of $\epsilon$, such that 
	\begin{equation}\label{eq:estfree}
		\|u\|_{L^q(\mathbb{R}^n)}=\|(\Delta+(1+i\epsilon))^{-1}F\|_{L^q(\mathbb{R}^n)}\leq C\|F\|_{L^p(\mathbb{R}^n)}
	\end{equation}
	when $(\frac{1}{p},\frac{1}{q})\in\Delta(n)$, $n\geq3$.
\end{theorem}

As we mentioned, another tool that will be crucial in the proof is an $L^2$-local estimate, which bounds the truncated Morrey-Campanato norm of the solution of the free equation, defined in \eqref{eq:triplenorma0}, in terms of the $L^p$ norm of the RHS data. This Theorem also appears in \cite{RV}, \cite{Gut} and \cite{Gut1}. The statement is the following.

\begin{theorem}\label{thm:triple}Let u be a solution of 
\begin{equation}
		\Delta u+(1+i\epsilon)u=-F,\quad \epsilon>0.
\end{equation}
If 
\begin{itemize}
\item[\textit{(i)}]
$n=3$ \text{or} $4$\quad\text{and} \quad$\frac{1}{n+1}\leq\frac{1}{p}-\frac{1}{2}<\frac{1}{2}$,\quad or
\item[\textit{(ii)}]
$n\geq5$\quad\text{and} \quad$\frac{1}{n+1}\leq\frac{1}{p}-\frac{1}{2}\leq\frac{2}{n}$,\quad
\end{itemize}
then, there exists a constant $C$, independent of $\epsilon$, such that
\begin{equation}\label{eq:tripletrun}
	\sup_{R\geq1}\left(\frac{1}{R}\int_{B_{R}}|u(x)|^2dx\right)^{1/2}\leq C\|F\|_{L^p(\mathbb{R}^n)}.
\end{equation}
\end{theorem}

The last result concerns an a priori estimate for the solution of the perturbed equation. It states that, given precise conditions, without assuming smallness, on the decay at infinity for the the electric potential, the magnetic potential and the radial derivative of the electric potential, we have a precise control for $\|\nabla_{A}u\|_{\frac{-(1+\delta)}{2}}$ and $\|u\|_{\frac{-(1+\delta)}{2}}$. This result appears in \cite{IS} .

The Theorem is the following one.

\begin{theorem}\label{thm:apriori}
Let $n\geq3$, and $u\in C_{0}^\infty$ be a solution of
\begin{equation*}
	(\nabla-iA(x))^2u+(1\pm i\epsilon)u-V(x)u=f,\quad \epsilon\neq0.
\end{equation*}
Let us assume that:
\begin{itemize}
\item[(V)]\label{V}
$V(x)$ can be decomposed as $V(x)=V_{1}(x)+V_{2}(x)$, and there exist strictly positive constants $C$ and $\mu$ such that
\item[$(V_{1})$]\label{V_{1}}
\begin{equation*}
|V_{1}(x)|\leq C|x|^{-\mu},\quad(\partial_{r}V_{1})(x)\leq C|x|^{-1-\mu},\quad|x|\geq1,
\end{equation*}
\item[$(V_{2})$]\label{V_{2}}
\begin{equation*}
|V_{2}(x)|\leq C|x|^{-1-\mu},\quad|x|\geq1,
\end{equation*}
\item[(B)]\label{B}
\begin{equation*}
|B(x)|\leq C|x|^{-1-\mu},\quad|x|\geq1.
\end{equation*}
\end{itemize}
Choose $\delta>0$ sufficiently small (so that $\delta\leq\mu/2$, $\delta<1$). Then, there exists a constant $C=C(\delta)$, which depends uniformly in $\epsilon$, such that the following a priori estimate holds
\begin{equation}
	\|\nabla_{A}u\|_{\frac{-(1+\delta)}{2}}+\|u\|_{\frac{-(1+\delta)}{2}} \leq C\|f\|_{\frac{1+\delta}{2}}.
\end{equation}
\end{theorem}
\begin{remark}\label{rm:singularities}
Observe that, since the electric potential $V$ and magnetic potential $A$ must satisfy the conditions of the theorem, singularities at the origin are not allowed.
\end{remark}

\begin{remark}
Notice that the unique continuation property holds for the differential operator $H=(\nabla-iA(x))^2-V(x)$, as can be seen in \cite{R}.
The assumptions $(V)$, $(V_{1})$, $(V_{2})$ and $(B)$, together with this observation implies that the limiting absorption principle holds.
\end{remark}

\begin{remark}The conditions on the decay for the electric potential $V$ and the magnetic field $B$ given by $(V_{1})$, $(V_{2})$ and $(B)$ respectively, are sufficient for us, due we have fixed the frequency $\tau=1$. It can be seen in \cite{IS}, that the result is true provided $\tau$ and $\epsilon$ belong to the following set denoted by $K$
\begin{equation}
K=\{k=\tau+i\epsilon\in\mathbb{C}/\tau\in(\tau_{0},\tau_{1}), \epsilon\in(0,\epsilon_{1})\},
\end{equation} 
where $0<\tau_{0}<\tau_{1}<\infty$ and $0<\epsilon_{1}<\infty$.

Hence, the critical case $\tau=0$ is excluded. This situation requires more decay for both potentials (typically $\xx^{-(2+\epsilon)}$, $\epsilon>0$, for $B$ and $V$ if $n=3$ and $\xx^{-2}$ for $n\geq4$, where $\xx=(1+|x|^2)^{1/2}$), in order to obtain a priori estimates for the solution of the perturbed equation, as can be seen in \cite{F}, where Morrey -Campanato type estimates, uniform in $\epsilon$, are obtained for $\tau\geq0$. 

Note also that this result gives an a priori estimate without assuming smallness neither for the non repulsive component of the electric field nor for the trapping component of the magnetic field, defined by $B_{\tau}:=\frac{x}{|x|}B$.
\end{remark}

Once we have described all the tools which are going to be used, it is necessary to introduce the region where we are able to extend the known results for the free Helmholtz equation to the case when electromagnetic perturbations are considered. During the discussion, it will appear a subregion of $\Delta(n)$, for $n\geq3$, which will be denoted by $\Delta_{0}(n)$, given by
\begin{equation}
	\Delta_{0}(n)=\left\{\left(\frac{1}{p},\frac{1}{q}\right)\in\Delta(n):\frac{1}{n+1}\leq \frac{1}{p}-\frac{1}{2},\frac{1}{n+1}\leq \frac{1}{2}-\frac{1}{q}\right\}.
\end{equation}
The set $\Delta_{0}(n)$ is the solid triangle determined by the points $Q$, $Q'$ and $Q''$ (see Figure \ref{fig:deltan0}).\\
This will be the region of boundedness for the perturbed Helmholtz equation.


\begin{figure}[!htb]
\begin{center}
\includegraphics[]{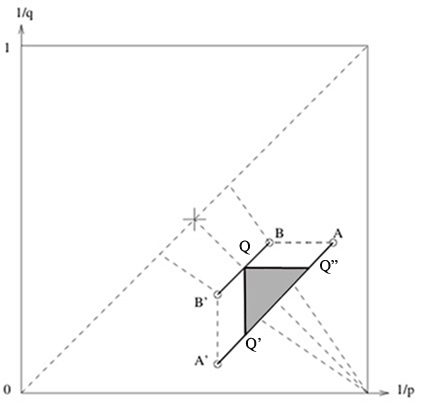}
\caption{$\Delta_{0}(n)$, $n\geq3$.}
\label{fig:deltan0}
\end{center}
\end{figure}

\begin{remark}
For the case of the perturbed electromagnetic equation we are not able to obtain a positive result of boundedness for the whole region $\Delta(n)$, since we have not control for the gradient term, namely $A\cdot\nabla$, outside $\Delta_{0}(n)$. However, when we set $A\equiv0$, and consider the electric hamiltonian, the results can be extended outside $\Delta(n)$ by imposing more decay on $V$.
\end{remark}


\section{Electromagnetic Helmholtz Equation.}\label{sec:magnetic}

In this section we will give the precise statement and the proof of the theorem, where we extend the known result for the free Helmholtz equation to the case when electromagnetic perturbations are considered. The basic theorems which will be used along the proof were given in the section \ref{sec:introd}, as well as the basic notation. First we will announce the result for the general electromagnetic case and afterwards, by setting $A\equiv0$, the electric case will be treated by extending our previous result.

Let us start by considering the solution of the Helmholtz equation with electromagnetic potential that satisfies either the ingoing or the outgoing Sommerfeld radiation condition. For $n\geq3$, it reads,

\begin{equation}\label{eq:Helmmagnetic}
	(\nabla-iA(x))^2u+u-V(x)u=f\quad \text{in}\quad \mathbb{R}^n,
\end{equation}
where $A:(A^1,\dots, A^n):\mathbb{R}^n\rightarrow\mathbb{R}^n$ is the magnetic potential and $V(x):\mathbb{R}^n\rightarrow\mathbb{R}$ is the electric potential.

We will assume that the magnetic potential $A$ satisfies the Coulomb gauge condition
\begin{equation}
 \nabla\cdot A=0.
\end{equation}

We will prove $L^{p}$-$L^{q}$ estimates for the solution of the equation \eqref{eq:Helmmagnetic}.
In order to do that we will consider the solution of \eqref{eq:Helmmagnetic} as the solution of the modified Helmholtz electromagnetic equation, 
\begin{equation}\label{eq:helma}
	(\nabla-iA(x))^2u+(1\pm i\epsilon)u-V(x)u=f\quad \text{in}\quad \mathbb{R}^n,\quad\epsilon\neq0,
\end{equation}
via limiting absorption principle, by taking the limit of the solution of \eqref{eq:helma} when $\epsilon$ goes to 0. We will obtain the corresponding $L^{p}$-$L^{q}$ estimates, independent of $\epsilon$, for the solution of \eqref{eq:helma}, so these will remain true for the solution of \eqref{eq:Helmmagnetic}. This is guaranteed  by the results appearing in \cite{IS}.

The goal is to determine the region of $p$ and $q$ where the solution of \eqref{eq:helma} satisfies $L^{p}$-$L^{q}$ estimates, namely, 
\begin{equation}
	\|u\|_{L^q(\mathbb{R}^n)}\leq C \|f\|_{L^p(\mathbb{R}^n)}.
\end{equation}
with $C$ independent of $\epsilon$.


The main result of this section is the following.


\begin{theorem}
\label{thm:elecmag}
\addcontentsline{toc}{subsection}{Theorem \ref{thm:elecmag}. $L^p$-$L^q$ estimates for Electromagnetic Helmholtz equation.}

Let u be a solution of 
\begin{equation}\label{eq:Helmmag}
	(\nabla-iA(x))^2u+(1\pm i\epsilon)u-V(x)u=f\quad \text{in}\quad \mathbb{R}^n,\quad n\geq3,\quad\epsilon\neq0.
\end{equation}
Let $V$ and $A$ satisfy $(V)$, $(V_{1})$, $(V_{2})$  and $(B)$  in Theorem \ref{thm:apriori}, and suppose that there exist constants $C,\mu>0$ such that
\begin{equation}\label{eq:potmag}
		|A(x)|\leq \frac{C}{(1+|x|)^{1+\mu}},\quad |V(x)|\leq \frac{C}{(1+|x|)^{1+\mu}}.
\end{equation}
Then, there exists a constant C, independent of $\epsilon$, such that 
	\begin{equation}
		\|u\|_{L^q(\mathbb{R}^n)}\leq C \|f\|_{L^p(\mathbb{R}^n)},
	\end{equation}
	when $\left(\frac{1}{p},\frac{1}{q}\right)\in\Delta_{0}(n)$.
\end{theorem}

\begin{proof}

\begin{remark}Notice that there are no smallness assumption neither for the electric potential $V$ nor for the magnetic potential $A$. Also we bound the solution only assuming short-range decay for $V$ and $A$. As we said in Remark \ref{rm:singularities}, singularities at the origin for $V$ and $A$ are not considered.
\end{remark}


\noindent{\bf Step 1.} It will be proved that, whenever $\left(\frac{1}{p},\frac{1}{q}\right)\in\Delta_{0}(n)$ then we get the following 
\begin{equation}
  \|u\|_{L^q(\mathbb{R}^n)}\leq C \|f\|_{\frac{1+\delta}{2}}.
\end{equation}
Let $u$ be a solution of \eqref{eq:Helmmag}. Since $\nabla\cdot A\equiv0$, we can expand the term $(\nabla-iA)^2$ in the following form
\begin{equation}
	(\nabla-iA)^2u=\Delta u-2iA\cdot\nabla_{A}u+|A|^2u.
\end{equation}
This is the key point in order to consider the electromagnetic case as a perturbation of the free equation. As can be seen, there appear terms of order zero and order one.
So by passing terms to the RHS, we have that $u$ is solution of the following equation
\begin{equation}
	\Delta u+(1\pm i\epsilon)u=f+2iA\cdot\nabla_{A}u-|A|^2u+Vu.
\end{equation}
Now we apply the result coming from Theorem \ref{thm:triple}. By considering the dual estimate of \eqref{eq:tripletrun}, we get that, if $\left(\frac{1}{p},\frac{1}{q}\right)\in\Delta_{0}(n)$ it holds
\begin{align}\label{eq:desNmag}
	\|u\|_{L^q(\mathbb{R}^n)}&=\|(\Delta+(1\pm i\epsilon))^{-1}(f+2iA\cdot\nabla_{A}u-|A|^2u+Vu)\|_{L^q(\mathbb{R}^n)}\\
	&\leq C(N_{0}(f)+N_{0}(2iA\cdot\nabla_{A}u)+N_{0}(|A|^2u)+N_{0}(Vu)).
	\nonumber
\end{align}
with $C$ independent of $\epsilon$ and $N_{0}$ defined in \eqref{eq:N0}.

Now we continue by treating the terms appearing on the RHS of \eqref{eq:desNmag}.
First we deal with the term $N_{0}(2iA\cdot\nabla_{A}u)$. From \eqref{eq:potmag}, we get that this term can be bounded as
\begin{align}
	N_{0}(2iA\cdot\nabla_{A}u)^2&=C\sum_{j\geq 0}2^{j+1}\int_{C(j)}|A\cdot\nabla_{A}u|^2dx\\
	&
	\leq C\sum_{j\geq 0}2^{j}\int_{C(j)}|A|^2|\nabla_{A}u|^2dx
	\nonumber
	\\
       	&
	\leq C \sum_{j\geq0}2^{j(\delta-2\mu)}\int_{\mathbb{R}^n}\frac{|\nabla_{A}u|^2}{(1+|x|)^{1+\delta}}dx.
	\nonumber
\end{align}
Therefore, we finally have
\begin{equation}\label{eq:despgrad}
	N_{0}(2iA\cdot\nabla_{A}u)\leq C\|\nabla_{A}u\|_{\frac{-(1+\delta)}{2}}.
\end{equation}

Let us continue with the term $N_{0}(|A|^2u)$. As before, from \eqref{eq:potmag}, we can treat this term as follows
\begin{align}
	N_{0}(|A|^2u)^2&=\sum_{j\geq 0}2^{j+1}\int_{C(j)}||A|^2u|^2dx\\
	&
	\leq C\sum_{j\geq 0}2^{j}\int_{C(j)}|A|^4|u|^2dx
	\nonumber
	\\
	&
	\leq C \sum_{j\geq0}2^{j(-2+\delta-4\mu)}\int_{\mathbb{R}^n}\frac{|u|^2}{(1+|x|)^{1+\delta}}dx.
	\nonumber
\end{align}
Hence, we get 
\begin{equation}\label{eq:despmod}
	N_{0}(|A|^2u)\leq C\|u\|_{\frac{-(1+\delta)}{2}}.
\end{equation}
The last term is $N_{0}(Vu)$. Similarly, we obtain from  \eqref{eq:potmag}
\begin{align}
	N_{0}(Vu)^2&=\sum_{j\geq 0}2^{j+1}\int_{C(j)}|Vu|^2dx\\
	&
	\leq C\sum_{j\geq 0}2^{j}\int_{C(j)}|V|^2|u|^2dx
	\nonumber
	\\
	&
	\leq C \sum_{j\geq0}2^{j(\delta-2\mu)}\int_{\mathbb{R}^n}\frac{|u|^2}{(1+|x|)^{1+\delta}}.dx
	\nonumber	
\end{align}
So, it verifies
\begin{equation}\label{eq:despele}
	N_{0}(Vu)\leq C\|u\|_{\frac{-(1+\delta)}{2}}.
\end{equation}
From \eqref{eq:desNmag}, \eqref{eq:despgrad}, \eqref{eq:despmod} and \eqref{eq:despele}, we get that, whenever $\left(\frac{1}{p},\frac{1}{q}\right)\in\Delta_{0}(n)$, the $L^q$ norm of $u$ can be bounded as
\begin{align}\label{eq:N0f}
	\|u\|_{L^q(\mathbb{R}^n)}&=\|(\Delta+(1\pm i\epsilon))^{-1}(f+2iA\cdot\nabla_{A}u-|A|^2u+Vu)\|_{L^q(\mathbb{R}^n)}\\
	&\leq C(N_{0}(f)+N_{0}(2iA\cdot\nabla_{A}u)+N_{0}(|A|^2u)+N_{0}(Vu)).
	\nonumber
	\\
	&\leq CN_{0}(f)+C_{1}\|\nabla_{A}u\|_{\frac{-(1+\delta)}{2}}+C_{2}\|u\|_{\frac{-(1+\delta)}{2}}.
	\nonumber
\end{align}

Now we remind the a priori estimate given by Theorem \ref{thm:apriori}, which ensures that, under the assumptions $(V)$, $(V_{1})$, $(V_{2})$ and $(B)$ for $V$ and $A$ respectively, there exists a constant $C$, such that the following holds
\begin{equation}\label{eq:f}
	\|\nabla_{A}u\|_{\frac{-(1+\delta)}{2}}+\|u\|_{\frac{-(1+\delta)}{2}} \leq C\|f\|_{\frac{1+\delta}{2}}.
\end{equation}
\begin{remark}The constant $C$ which appears here depends uniformly in $\epsilon$.
\end{remark}
\noindent So, from \eqref{eq:N0f} and \eqref{eq:f}, we get
\begin{equation}\label{eq:N0ff}
\|u\|_{L^q(\mathbb{R}^n)}\leq CN_{0}(f)+C_{1}\|f\|_{\frac{1+\delta}{2}}.
\end{equation}
Moreover, it holds that $N_{0}(f)$ can be bounded as
\begin{equation}
	N_{0}(f)\leq C\|f\|_{\frac{1+\delta}{2}}.
\end{equation}
This, together with \eqref{eq:N0ff} concludes that, if $u$ is a solution of \eqref{eq:Helmmag}, it verifies
\begin{equation}
	 \|u\|_{L^q(\mathbb{R}^n)}\leq C\|f\|_{\frac{1+\delta}{2}}.
\end{equation}
Therefore, we get the desired estimate.

\noindent{\bf Step 2.} By applying duality to the last estimate, we get that if $\left(\frac{1}{p},\frac{1}{q}\right)\in\Delta_{0}(n)$,
\begin{equation}\label{eq:step2}
	\|u\|_{\frac{-(1+\delta)}{2}}\leq C \|f\|_{{L^p(\mathbb{R}^n)}}
\end{equation}
\begin{remark}
The adjoint operator is the one corresponding to $\mp \epsilon$. Since we can do the same argument for both signs, all the computations are justified.
\end{remark}

\noindent{\bf Step 3.} This is the final step in the proof. As we said in the introduction the main difficulty will be to handle the first order term given by $A\cdot\nabla_{A}u$, since there are no $L^p$-$L^q$ estimates for the gradient of the solution of the free Helmhotz equation. Instead of considering this norm our argument will end up by treating $\|\nabla_{A}u\|_{\frac{-(1+\delta)}{2}}$, and this norm will be under control in the region $\Delta_{0}(n)$.

Consequently, we have that if $\left(\frac{1}{p},\frac{1}{q}\right)\in\Delta_{0}(n)$, from the $L^{p}$-$L^{q}$ estimates for the solution of the free equation, namely \eqref{eq:estfree}, given in Theorem \ref{thm:Helmfree}, and proceeding as in the step 1 for the terms $2iA\cdot\nabla_{A}u$, $|A|^2u$ and $Vu$, we get
\begin{align}
	\|u\|_{L^q(\mathbb{R}^n)}&=\|(\Delta+(1\pm i\epsilon))^{-1}(f+2iA\cdot\nabla_{A}u-|A|^2u+Vu)\|_{L^q(\mathbb{R}^n)}\\
	&\leq C\|f\|_{L^p(\mathbb{R}^n)}+C_{1}(N_{0}(2iA\cdot\nabla_{A}u)+N_{0}(|A|^2u)+N_{0}(Vu)).
	\nonumber
\end{align}
where $C$ and $C_{1}$ do not depend on $\epsilon$.

Let us remind that, from \eqref{eq:potmag} it holds
\begin{align}
&N_{0}(2iA\cdot\nabla_{A}u)\leq C_{1}\|\nabla_{A}u\|_{\frac{-(1+\delta)}{2}},
\\
&
N_{0}(|A|^2u)\leq C_{2}\|u\|_{\frac{-(1+\delta)}{2}},
\nonumber
\\
&
N_{0}(Vu)\leq C_{3}\|u\|_{\frac{-(1+\delta)}{2}}.
\nonumber
\end{align}
Therefore, we get
\begin{equation}
	\|u\|_{L^q(\mathbb{R}^n)}\leq C\|f\|_{L^p(\mathbb{R}^n)}+C_{1}\|\nabla_{A}u\|_{\frac{-(1+\delta)}{2}}+C_{2}\|u\|_{\frac{-(1+\delta)}{2}}.
\end{equation}
Finally we conclude that if $u$ is solution of \eqref{eq:Helmmag}, by applying \eqref{eq:step2} we can bound its $L^q$ norm as follows
\begin{equation}\label{eq:grad}
	\|u\|_{L^q(\mathbb{R}^n)}\leq C\|f\|_{L^p(\mathbb{R}^n)}+C_{1}\|\nabla_{A}u\|_{\frac{-(1+\delta)}{2}}.
\end{equation}
It remains to bound $\|\nabla_{A}u\|_{\frac{-(1+\delta)}{2}}$. This can be done in the following way. Let us consider a radial function $\varphi=\varphi(|x|)\in C_{0}^\infty$. By multiplying Helmholtz equation \eqref{eq:Helmmag} by $\varphi\bar{u}$ in the $L^2$-sense and taking the resulting real parts it gives the identity
\begin{align}\label{eq:identity}
-&\int_{\mathbb{R}^{n}}\varphi|\nabla_{A}u|^2dx+\frac{1}{2}\int_{\mathbb{R}^{n}}\Delta\varphi|u|^2dx-\int_{\mathbb{R}^{n}}\varphi V|u|^2dx+\int_{\mathbb{R}^{n}}\varphi|u|^2dx\\
&=\Re\int_{\mathbb{R}^{n}} f\varphi\bar{u}dx.
\nonumber
\end{align}
From this, passing some terms to the RHS and taking modulus, it holds
\begin{align}\label{eq:identity}
&\int_{\mathbb{R}^{n}}\varphi|\nabla_{A}u|^2dx\leq\frac{1}{2}\int_{\mathbb{R}^{n}}|\Delta\varphi||u|^2dx+\int_{\mathbb{R}^{n}}|\varphi| |V||u|^2dx+\int_{\mathbb{R}^{n}}|\varphi||u|^2dx\\
&+\int_{\mathbb{R}^{n}} |f\varphi\bar{u}|dx.
\nonumber
\end{align}

Assume for simplicity that $f$ has compact support and therefore from {\bf Step 1} we have that
\begin{equation}
	 \|u\|_{L^q(\R^n)}<+\infty. 
 \end{equation}
Then we proceed by density.

Now we pass to choose the appropriate function $\varphi$. Let us consider $\varphi=\varphi(|x|)=\frac{1}{(1+|x|)^{1+\delta}}$.\\
Now, we start to bound some of the terms appearing in the RHS of \eqref{eq:identity}. We trivially have
\begin{equation}
|\Delta\varphi|\leq C|\varphi|,\qquad
\end{equation}
so it holds

\begin{equation}
\int_{\mathbb{R}^{n}}|\Delta\varphi||u|^2dx\leq C\int_{\mathbb{R}^{n}}|\varphi||u|^2dx,
\end{equation}
and from \eqref{eq:potmag}, we obtain
\begin{equation}
 \int_{\mathbb{R}^{n}}|\varphi| |V||u|^2dx\leq C\int_{\mathbb{R}^{n}}|\varphi||u|^2dx.
\end{equation}
Then, we conclude
\begin{equation}\label{eq:ineR}
\int_{\mathbb{R}^{n}}\frac{|\nabla_{A}u|^2}{(1+|x|)^{1+\delta}}dx\leq C\int_{\mathbb{R}^{n}}\frac{|u|^{2}}{(1+|x|)^{1+\delta}}dx+\int_{\mathbb{R}^{n}}|f\varphi\bar{u}|dx
\end{equation}
For the last term, by applying H\"older inequality with $1=\frac{1}{p}+\frac{1}{q}+\frac{1}{r}$, we can bound it as
\begin{equation}\label{eq:hol}
	\int_{\mathbb{R}^n}|f\varphi\bar{u}|dx\leq\|f\|_{L^{p}}\|u\|_{L^{q}}\|\varphi\|_{L^r},\\
	\end{equation}
where $r$ is to be determined. From \eqref{eq:ineR} and \eqref{eq:hol} we obtain that
\begin{equation}
\int_{\mathbb{R}^{n}}\frac{|\nabla_{A}u|^2}{(1+|x|)^{1+\delta}}dx\leq C\int_{\mathbb{R}^{n}}\frac{|u|^{2}}{(1+|x|)^{1+\delta}}dx+C_{1}\|f\|_{L^{p}}\|u\|_{L^{q}}.
\end{equation}
This leads to the crucial estimate
\begin{equation}\label{eq:norms}
\|\nabla_{A}u\|_{\frac{-(1+\delta)}{2}}\leq C\|u\|_{\frac{-(1+\delta)}{2}}+\alpha\|u\|_{L^{q}}+C(\alpha)\|f\|_{L^{p}},
\end{equation}
for $\alpha,C(\alpha)>0$.

Therefore, we have to find the region for $p$ and $q$, for $(\frac{1}{p},\frac{1}{q})\in\Delta_{0}(n)$, where the following conditions are satisfied
\begin{equation}\label{eq:conditions}
\frac{n}{r}\leq1,\qquad1=\frac{1}{p}+\frac{1}{q}+\frac{1}{r}.
\end{equation}
We have that \eqref{eq:conditions} holds if $(\frac{1}{p},\frac{1}{q})\in\Delta_{0}^-(n)$, where $\Delta_{0}^-(n)$ is given by
\begin{equation}
	\Delta_{0}^-(n)=\left\{\left(\frac{1}{p},\frac{1}{q}\right)\in\Delta_{0}(n):\frac{1}{q}\leq 1-\frac{1}{p}\right\}.
\end{equation}
The region $\Delta_{0}^-(n)$ is determined by the points $\left(\frac{1}{p},\frac{1}{q}\right)\in\Delta_{0}(n)$ under the diagonal $1/q=1-1/p$, including the points $\left(\frac{1}{p},\frac{1}{q}\right)$ in the duality line (see Figure \ref{fig:deltan0med1}).

\begin{figure}[!htb]
\begin{center}
\includegraphics[]{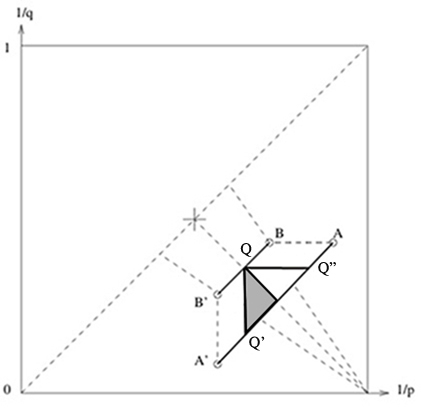}
\caption{$\Delta_{0}^-(n)$, $n\geq3$.}
\label{fig:deltan0med1}
\end{center}
\end{figure}
Finally we have that if $(\frac{1}{p},\frac{1}{q})\in\Delta_{0}^-(n)$, from \eqref{eq:step2} and \eqref{eq:norms} it holds
\begin{equation}
\|\nabla_{A}u\|_{\frac{-(1+\delta)}{2}}\leq C\|f\|_{L^{p}}+\alpha\|u\|_{L^{q}}+C(\alpha)\|f\|_{L^{p}},
\end{equation}
for $\alpha, C(\alpha)>0$. This, together with \eqref{eq:grad} leads to
\begin{equation}
	\|u\|_{L^q(\mathbb{R}^n)}\leq C\|f\|_{L^p(\mathbb{R}^n)}+\alpha\|u\|_{L^{q}}+C(\alpha)\|f\|_{L^{p}},
\end{equation}
for $(\frac{1}{p},\frac{1}{q})\in\Delta_{0}^-(n)$. Now we choose $\alpha$ sufficiently small and conclude for $(\frac{1}{p},\frac{1}{q})\in\Delta_{0}^-(n)$ the desired bound
\begin{equation}
		\|u\|_{L^q(\mathbb{R}^n)}\leq C \|f\|_{L^p(\mathbb{R}^n)}.
\end{equation}
Finally, by applying duality we have that the result is true for $(\frac{1}{p},\frac{1}{q})\in\Delta_{0}(n)$. The proof is complete.
\end{proof}

Notice that the region of boundedness for the solution of the Helmholtz equation with electromagnetic potential is smaller than the region obtained in the free case. This is due to the presence of the first order term. Therefore it is natural to consider if, whenever $A\equiv0$ (i.e., we deal with the Helmholtz equation with electric potential), this region can be extended to the whole $\Delta(n)$. This will be our next aim.

This idea is resumed in the following result, where we extend the boundedness of the solution to the whole $\Delta(n)$. We have been able to prove the next Theorem.


\begin{theorem}
\label{thm:extension}
\addcontentsline{toc}{subsection}{Theorem \ref{thm:extension}. $L^p$-$L^q$ estimates for Electric Helmholtz equation.}

Let u be a solution of 
	\begin{equation}\label{eq:Helmelec}
		-\Delta u+(1\pm i\epsilon)u+V(x)u=f,\quad \text{in}\quad \mathbb{R}^n,\quad n\geq3,\quad\epsilon\neq0.
	\end{equation}
	If $V$ satisfies $(V)$, $(V_{1})$ and $(V_{2})$ in Theorem \ref{thm:apriori}, and suppose that there exist constants $C,\gamma>0$ such that
	\begin{equation}
		|V(x)|\leq \frac{C}{(1+|x|)^{\gamma+\mu}},
	\end{equation}
	and $\gamma$ satisfies for $\left(\frac{1}{p},\frac{1}{q}\right)\in\Delta(n)\setminus\Delta_{0}(n)$,
	\begin{equation}\label{eq:gamma}
	\gamma >
	\begin{cases}
	\frac{1+\delta}{2}-\mu+n\left\{\frac{2}{n+1}-(\frac{1}{2}-\frac{1}{q})\right\},\quad\frac{1}{2n}<\frac{1}{2}-\frac{1}{q}<\frac{1}{n+1},\\
	\frac{1+\delta}{2}-\mu+n\left\{\frac{2}{n+1}-(\frac{1}{p}-\frac{1}{2})\right\},\quad\frac{1}{2n}<\frac{1}{p}-\frac{1}{2}<\frac{1}{n+1},
	\end{cases}
	\end{equation}
	then, there exists a constant C, independent of $\epsilon$, such that 
	\begin{equation}
		\|u\|_{L^q(\mathbb{R}^n)}\leq C \|f\|_{L^p(\mathbb{R}^n)}.
	\end{equation}

\end{theorem}
\begin{proof}
As for the proof of Theorem \ref{thm:elecmag}, this will be divided in three steps. The first two steps are the same, so we will skip them. The main difference appears in Step 3.
\begin{remark}\label{rm:proof}
As we have seen in Theorem \ref{thm:elecmag}, whenever $\left(\frac{1}{p},\frac{1}{q}\right)\in\Delta_{0}(n)$, the decay assumption \eqref{eq:potmag} for the electric potential $V$ is sufficient ir order to prove the $L^{p}$-$L^{q}$ estimates. As we will see, outside this region, more decay for the electric potential $V$ potential is needed. 
\end{remark}

\noindent{\bf Step 3.} Now let us consider $\left(\frac{1}{p},\frac{1}{q}\right)\in\Delta(n)$ such that $\frac{1}{2n}<\frac{1}{2}-\frac{1}{q}<\frac{1}{n+1}$. Then from Theorem \ref{thm:Helmfree} and observing that, since the dual estimate of \eqref{eq:tripletrun} in Theorem \ref{thm:triple} can not be applied for the perturbative term $Vu$ because we are outside the allowed range for $q$, we conclude that the $L^q$ norm of the solution of the equation \eqref{eq:Helmelec} can be bounded as follows 
\begin{align}\label{eq:Vu}
	\|u\|_{L^q(\mathbb{R}^n)}&=\|(-\Delta+(1\pm i\epsilon))^{-1}(f-Vu)\|_{L^q(\mathbb{R}^n)}\\
	&
	\leq C(\|f\|_{{L^p(\mathbb{R}^n)}}+\|Vu\|_{{L^{p_{1}}(\mathbb{R}^n)}}).
	\nonumber
\end{align}
where $C$ does not depend on $\epsilon$.\\
Here $p_{1}$ is given by
\begin{equation}
	\frac{1}{p_{1}}-\frac{1}{q}=\frac{2}{n+1},\quad\frac{1}{2n}<\frac{1}{2}-\frac{1}{q}<\frac{1}{n+1}.
\end{equation}
We have taken the point $p_{1}$ being in the line $1/p_{1}-1/q=2/(n+1)$ in order the require the smallest decay for $V$.

Now, we can bound the term $Vu$ in \eqref{eq:Vu} as follows. By applying H\"older inequality we trivially get
\begin{align}\label{eq:holder}
\|Vu\|_{{L^{p_{1}}(\mathbb{R}^n)}}&\leq C\|(1+|x|)^{-(\gamma+\mu)+\frac{1+\delta}{2}}u(1+|x|)^{-\frac{1+\delta}{2}}\|_{{L^{p_{1}}(\mathbb{R}^n)}}\\
&\leq C\|(1+|x|)^{-(\gamma+\mu)+\frac{1+\delta}{2}}\|_{{L^{r}(\mathbb{R}^n)}}\|u(1+|x|)^{-\frac{1+\delta}{2}}\|_{{L^{2}(\mathbb{R}^n)}}.
\nonumber
\end{align}
where $r$ is given by
\begin{equation}
	\frac{1}{p_{1}}=\frac{2}{n+1}+\frac{1}{q}=\frac{2}{n+1}-\left(\frac{1}{2}-\frac{1}{q}\right)+\frac{1}{2}=\frac{1}{r}+\frac{1}{2}.
\end{equation}
We have that
\begin{equation}
\|(1+|x|)^{-(\gamma+\mu)+\frac{1+\delta}{2}}\|_{{L^{r}(\mathbb{R}^n)}}<\infty,
\end{equation}
provided that
\begin{equation}
	\gamma>\frac{1+\delta}{2}-\mu+n\left\{\frac{2}{n+1}-(\frac{1}{2}-\frac{1}{q})\right\}.
\end{equation}

Finally, from \eqref{eq:Vu}, \eqref{eq:holder} and reminding that $\|u\|_{\frac{-(1+\delta)}{2}}$ is still bounded for  $\left(\frac{1}{p},\frac{1}{q}\right)\in\Delta(n)$ such that $\frac{1}{2n}<\frac{1}{2}-\frac{1}{q}<\frac{1}{n+1}$, we can conclude the final estimate
\begin{equation}
		\|u\|_{L^q(\mathbb{R}^n)}\leq C \|f\|_{L^p(\mathbb{R}^n)}.
\end{equation}
Now, by applying duality we have that the result is true whenever $\left(\frac{1}{p},\frac{1}{q}\right)\in\Delta(n)$ such that $\frac{1}{2n}<\frac{1}{p}-\frac{1}{2}<\frac{1}{n+1}$.
This ends the proof.
\end{proof}
\begin{remark}
Notice that, from \eqref{eq:gamma}, the necessary decay $\gamma$ for $V$ in order to obtain the result grows as we approach the upper frontier of $\Delta(n)$ (See line segment $AB$ Figure \ref{fig:deltan0med1}). However, we can always take $\gamma<2$ and $|V(x)|\leq C/(1+|x|)^\gamma$, where $C$ is  not necessarily small.
\end{remark}

\end{document}